\documentclass[]{amsart}

\usepackage{mathrsfs}
\usepackage{graphicx}

\usepackage{amssymb}
\usepackage{amsmath}
\usepackage{amsthm}
\usepackage{galois}
\usepackage[all,cmtip]{xy}

\allowdisplaybreaks                 
\usepackage{float}                 
\restylefloat{figure}              
\usepackage{array}              
\newtheorem{thm}{Theorem}[section]

\newtheorem{lem}[thm]{Lemma}

\theoremstyle{definition}
\newtheorem{defn}[thm]{Definition}

\newtheorem{ques}[thm]{Question}
\theoremstyle{remark}
\newtheorem{rem}[thm]{Remark}

\numberwithin{equation}{section}

\newcommand{\Z}{\mathbb Z}

\newcommand{\fix}{\mathrm{Fix}}
\newcommand{\ind}{\mathrm{ind}}

\newcommand{\F}{\mathbf{F}}      

\newcommand{\B}{\mathcal{B}}

\UseRawInputEncoding
\begin{document}

\title{Aspherical manifolds which do not have Bounded Index Property}
\author{Qiang Zhang, Xuezhi Zhao}

\address{School of Mathematics and Statistics, Xi'an Jiaotong University, Xi'an 710049, China}
\email{zhangq.math@mail.xjtu.edu.cn}

\address{School of Mathematical Sciences, Capital Normal University, Beijing 100048, China}
\email{zhaoxve@mail.cnu.edu.cn}

\thanks{The authors are partially supported by NSFC (Nos. 11961131004, 11971389, 12071309 and 12271385).}



\date{\today}
\begin{abstract}
We prove that some aspherical manifolds do not have BIP, negating a question of Jiang.
\end{abstract}
\maketitle


\section{Introduction}

For a selfmap $f$ of a space $X$, Nielsen fixed point theory is concerned with
the properties of the fixed point set
$$\fix f:=\{x\in X\mid f(x)=x\}$$
that splits into a disjoint union of \emph{fixed point classes}. For each fixed point class $\F$ of $f$, a homotopy invariant \emph{index} $\ind(f,\F)\in \Z$ is well-defined (see Section \ref{sec. fpc} or \cite{J83} for definitions).

In 1998, Jiang \cite{J98} introduced the following definitions of BIP and BIPH, while Ye and Zhang \cite{ZY20} extended the definition of BIPHE.

\begin{defn}
A compact polyhedron $X$ is said to have the \emph{Bounded Index Property (BIP)}(resp. \emph{Bounded Index Property for Homeomorphisms (BIPH)}, \emph{Bounded Index Property for Homotopy Equivalences (BIPHE)}) if there is an integer $\B>0$ such that for any map (resp. homeomorphism, homotopy equivalence)$f: X\rightarrow X$ and any fixed point class $\F$ of $f$, the index $|\ind(f,\F)|\leq \B$.
\end{defn}

Moreover, Jiang asked the following question:

\begin{ques}(\cite[Qusetion 3]{J98})\label{Jiang's question}
Does every compact aspherical polyhedron $X$ (i.e. $\pi_i(X)=0$ for all $i>1$) have BIP or BIPH?
\end{ques}

\begin{rem}
Clearly, if $X$ has BIP, then it has BIPHE and hence has BIPH. For a closed aspherical manifold $M$, if the well-known Borel's conjecture (any homotopy equivalence $f:M\rightarrow M$ is homotopic to a homeomorphism $g:M\rightarrow M$) is true, then $M$ has BIPHE if and only if it has BIPH.
\end{rem}

Many types of aspherical polyhedra had been showed to support Question \ref{Jiang's question} with a positive answer:
infra-solvmanifolds have BIP \cite{Mc92}; graphs and hyperbolic surfaces have BIP \cite{J98, K97, K00, JWZ11};  geometric 3-manifolds have BIPH \cite{JW}; orientable Seifert 3-manifolds with hyperbolic orbifolds have BIPH \cite{Z12}; products of hyperbolic surfaces have BIPH \cite{ZZ19};  and products of negatively curved Riemannian manifolds have BIPHE \cite{Z14, Z15, ZY20}.

In this paper, we prove some aspherical manifolds that do not have BIP/BIPH. Let $S^1$ denote the circle, and let $\Sigma_2=T^2\sharp T^2$ denote the connected sum of two copies of the torus $T^2$, i.e., the connected orientable closed surface of genus 2. Our main results are the following.

\begin{thm}\label{main thm1}
$\Sigma_2\times S^1$ has BIPH, but does not have BIP.
\end{thm}

\begin{thm}\label{main thm2}
$\Sigma_2\times T^2$ does not have BIPH, and hence does not have BIP.
\end{thm}

In \cite{ZY20}, Ye and Zhang provided a sufficient condition for products to preserve BIP. Now, our results above indicate that BIP is not preserved under products in general.
To prove them, we construct a family of maps whose fixed point classes have indices going to infinity, see Section \ref{Proof of Theorem main thm1} and Section \ref{Proof of Theorem main thm2}.
The idea of our construction is partly inspired by Gogolev and Lafont \cite{GL16} and Neofytidis \cite{N20}.

\section{Facts on fixed point classes}\label{sec. fpc}

To prove the main theorems, we first introduce some facts on fixed point classes in this section, see \cite{J83} for more details.

\subsection{Fixed point class}

Let $f:X\rightarrow X$ be a selfmap of a compact polyhedron $X$, and let $q: \widetilde{X}\to X$ be a universal covering of $X$, with group $\pi$ of covering transformations which is identified with the fundamental group $\pi_1(X)$.
For any lifting $\tilde{f}:\widetilde{X}\to \widetilde{X}$ of $f$, the projection of its fixed point set is called a \emph{fixed
point class} of $f$, written $\F=q(\fix\tilde{f})$. Strictly speaking, we say two liftings $\tilde{f}$ and $\tilde{f'}$ of $f$ are \emph{conjugate} if there exists $\gamma\in \pi$ such that
$$\tilde{f'}=\gamma\comp \tilde{f}\comp \gamma^{-1}.$$
Then $\F=q(\fix\tilde{f})=q(\fix\tilde{f'})$ is said to be the fixed point class of $f$ \emph{labeled} by the conjugacy class $[\tilde f]$ of $\tilde{f}$. Thus, a fixed point class always carries a label which is a conjugacy class of liftings. The fixed point set $\fix f$ decomposes into a disjoint union of fixed point classes. When $\fix\tilde{f}= \emptyset$, $\F=q(\fix\tilde{f})$ is called an \emph{empty} fixed point class.

Moreover, each lifting $\tilde{f}$ induces an endomorphism $\tilde{f}_{\pi}:\pi\to\pi$ defined by
$$\tilde{f}\comp \gamma=\tilde{f}_{\pi}(\gamma)\comp \tilde{f},\quad \gamma\in \pi.$$
From now on, let a reference lifting $\tilde f$ of $f$ be chosen. Then every lifting of $f$ can be uniquely written as $\beta\comp \tilde f$ for some $\beta\in \pi$. Two liftings $\beta\comp \tilde f$ and $\beta'\comp \tilde f$ are conjugate, i.e., label the same fixed point class of $f$, if and only if $\beta, \beta'\in \pi$ are $\tilde f_\pi$-\emph{conjugate}, i.e., there exists $\gamma\in \pi$ such that
$$\beta'=\gamma^{-1}\beta\tilde f_\pi(\gamma).$$

A fixed point class is \emph{essential} if its index is non-zero. The number of essential fixed point classes of $f$ is called the $Nielsen$ $number$ of $f$, denoted by $N(f)$. The famous Lefschetz-Hopf theorem says that the sum of the indices of all the fixed points of $f$ is equal to the $Lefschetz$ $number$ $L(f)$.

\subsection{Product formula for the index of a fixed point class }

\begin{lem}\label{lem}
Let $p: E=\Sigma_2\times T^k\to \Sigma_2$ be the projection to the first factor, where $T^k$ ($k\geq 1$) is a $k$-torus. Let $f:E\to E$ be a fiber-preserving map with induced self-map $\bar f$ on the base space.
$$
\xymatrix{
E \ar[r]^f\ar[d]_p & E \ar[d]^p\\
\Sigma_2 \ar[r]^{\bar f} & \Sigma_2
}
$$
Then, for any fixed point class $\F$ of $f$, the projection $p(\F)$ is a fixed point class of $\bar f$, and
$$|\ind(f,\F)|=[\fix\bar f_{\pi} : p_{\pi}(\fix f_{\pi})]\cdot|\ind(\bar f, p(\F)|,$$
where $f_{\pi}: \pi_1(E, e)\to\pi_1(E,e)$ and $\bar f_{\pi}: \pi_1(\Sigma_2, x)\to\pi_1(\Sigma_2,x)$ for $e=(x,y)\in \F$, are the natural homomorphisms induced by $f$ and $\bar f$ respectively.
\end{lem}

\begin{proof}
Note that for any map of $T^k$, its essential fixed point classes have the same index $\pm 1$ \cite{B75, H79}. Then the conclusion follows from \cite[IV 1.6 and IV 3.3 Theorems]{J83}.
\end{proof}

\section{Proof of Theorem \ref{main thm1}}\label{Proof of Theorem main thm1}

\begin{proof}[\textbf{Proof of Theorem \ref{main thm1}}]
Let $E=\Sigma_2\times S^1$. Note that $E$ can be regarded as a Seifert 3-manifold with base space $\Sigma_2$, then it has BIPH following from \cite{Z12}.

Now we prove that it does not have BIP.

First, taking the classical presentations $\pi_1(S^1)  = \langle c \mid - \rangle,$
and
$$
\pi_1(\Sigma_2)
=\langle a_1, b_1, a_2, b_2 \mid a_1b_1a^{-1}_1b^{-1}_1a_2b_2a^{-1}_2b^{-1}_2=1 \rangle.
$$
For any integer $m\geq 1$, let $\tilde f: \mathbb{H}^2\times \mathbb{R}\to \mathbb{H}^2\times \mathbb{R}$ be a self-map defined by
$$\tilde f(z, x) = (z,\ \tilde r(z)+(m+1)x),$$
where $\tilde r: \mathbb{H}^2\to \mathbb{R}$ is a lifting of the retraction $r: \Sigma_2\to S^1$ determined by
the homomorphism $r_\pi: \pi_1(\Sigma_2)\to \pi_1(S^1)$, which is given by
\begin{equation}\label{eq-lift1}
  r_\pi(a_1) = c, \ r_\pi(b_1)=r_\pi(a_2)=r_\pi(b_2)=1.
\end{equation}

It is obvious that $\tilde f$ induces a fibre-preserving map $f: E\to E = \Sigma_2\times S^1$, and there is a commutative diagram
$$
\xymatrix{
E \ar[r]^f\ar[d]_p & E \ar[d]^p \\
\Sigma_2 \ar[r]^{\bar f} & \Sigma_2.
}
$$
Moreover, the induced self-map $\bar f$ on the base space $\Sigma_2$ is the identity. On each fibre $p^{-1}(z_0)=S^1$, the restriction $f|p^{-1}(z_0)$ of $f$ is a self-map of degree $m+1$, and hence has Lefschetz number $-m$. By \cite[IV 3.2 Theorem]{J83}, we have that
\begin{equation}\label{eq-Lef1}
  L(f) = (-m)\cdot L(\bar f) = (-m)\cdot \chi(\Sigma_2) =2m.
\end{equation}

Clearly, the lifting $\tilde f$ of $f$, induces an endomorphism $\tilde f_\pi$ on $\pi_1(E) = \pi_1(\Sigma_2)\times \pi_1(S^1)$ given by
$$
 \tilde f_\pi(a_1) = (a_1, c),\
 \tilde f_\pi(b_1) = b_1,\
 \tilde f_\pi(a_2) = a_2,\
 \tilde f_\pi(b_2) = b_2,\
 \tilde f_\pi(c)   = c^{m+1}.
$$
which is an injective endomorphism, but is not surjective.

Note that all liftings of $f$ are of the form $(u, c^s)\tilde f$, where $(u,c^s)\in \pi_1(\Sigma_2)\times \pi_1(S^1)$ is regarded as the deck transformation on the universal covering space $\widetilde E=\mathbb{H}^2\times \mathbb{R}$ of $E$. It is obvious that $\fix((u, c^s)\tilde f) = \emptyset$ if $u\ne 1$. Moreover, for any $s, s'\in \mathbb{Z}$,
$$(a_1^{s'-s}, 1)^{-1} (1, c^{s}) \tilde f_\pi(a_1^{s'-s},1)=(a_1^{s-s'}, 1) (1, c^{s})(a_1^{s'-s}, ~c^{s'-s})=(1, c^{s'}).$$
Thus, $f$ has unique non-empty fixed point class $\F$ determined by $\tilde f$, which has index
$$\ind(f, \F)=2m$$
by Equation (\ref{eq-Lef1}). Since the integer $m$ can be arbitrarily large, we have proven that $E$ does not have BIP.
\end{proof}

\begin{rem}
In the above proof, by a direct computation, we have
$$
\begin{array}{rcl}
\fix \tilde f
& = & \{(z, x)\in \mathbb{H}^2\times \mathbb{R}\mid (z, x) = \tilde f(z,x)  \}\\
& = & \{(z, x)\in \mathbb{H}^2\times \mathbb{R}\mid (z, x) = (z, \tilde r(z)+(m+1)x) \}\\
& = & \{(z, x)\in \mathbb{H}^2\times \mathbb{R}\mid  x=-\frac{\tilde r(z)}{m} \}.
\end{array}
$$
This is a connected subset homeomorphic to $\mathbb{H}^2$, and therefore the fixed point set of $f$ is an embedding surface in $E$. Note that the restriction of $p:E\to \Sigma_2$ on $\fix f$ is an $m$-sheet covering. We obtain  that $\F=\fix f$ is actually an orientable closed surface of genus $m+1\geq 2$.

In fact, if $m\leq -1$, self-maps of $E$ can be constructed similarly, see \cite[Example 5.1]{Z12} for the case $m=-1$.
\end{rem}

\section{Proof of Theorem \ref{main thm2}}\label{Proof of Theorem main thm2}

\begin{proof}[\textbf{Proof of Theorem \ref{main thm2}}]
Let $E=\Sigma_2\times T^2$, and let $p: E\to \Sigma_2$ be the projection to the first factor. Note that
$\pi_1(T^2)=\langle a, b \mid ab=ba \rangle,$
$$\pi_1(\Sigma_2)=\langle a_1, b_1, a_2, b_2 \mid a_1b_1a^{-1}_1b^{-1}_1a_2b_2a^{-1}_2b^{-1}_2=1 \rangle,$$
and
$$\pi_1(E)=\pi_1(\Sigma_2)\times \pi_1(T^2).$$

Now, let's take a few steps to prove the conclusion.

(1) For any integer $m\ge 1$, since $T^2$, $\Sigma_2$ and $E$ are aspherical, there is a fiber-preserving homeomorphism $f: E\to E$, $(x,y)\mapsto (x,~f_x(y))$ (for the detailed definition, see Step (2) below) with induced automorphism  $f_\pi: \pi_1(E)\to \pi_1(E)$ defined by
$$\pi_1(\Sigma_2)\times \pi_1(T^2)\ni (u, v)\stackrel{f_\pi}{\mapsto}(u, r_\pi(u)\xi(v)),$$
where $r_\pi: \pi_1(\Sigma_2)\to \pi_1(T^2)$ is given by
\begin{equation}\label{fuza chanrao}
r_\pi(a_1) = a,\ r_\pi(b_1) = b,\ r_\pi(a_2) = 1,\ r_\pi(b_2) = 1,\
\end{equation}
and $\xi: \pi_1(T^2)\to \pi_1(T^2)$ is an automorphism given by
$$\xi(a) = a^{m+1}b, \ \  \xi(b) = a^{m}b.
$$
Actually, $f_\pi$ is an automorphism with inverse given by
$$\pi_1(\Sigma_2)\times \pi_1(T^2)\ni (u, v){\mapsto}(u,\ \xi^{-1}(r_\pi(u^{-1}) v)).$$

(2) Construction of $f$. The fiber-preserving homeomorphism $f: E\to E$
$$
\xymatrix{
E \ar[r]^f\ar[d]_p & E \ar[d]^p \\
\Sigma_2 \ar[r]^{\bar f} & \Sigma_2
}
$$
is given by the lifting $\tilde f: \mathbb{H}^2\times \mathbb{R}^2\to \mathbb{H}^2\times \mathbb{R}^2$, which is defined by
$$(z, ~{{x}\choose{y}}) \stackrel{\tilde f}{\mapsto} (z, ~\tilde r(z)+
\left(
\begin{array}{ll}
  m+1 & m \\
  1 & 1  \\
\end{array}
\right)
{{x}\choose{y}}),$$
where $\tilde r: \mathbb{H}^2\to \mathbb{R}^2$ is the based lifting of the pinch map from $\Sigma_2$ to $T^2$ determined by $r_\pi$.\\

(3) Below we show that $\F=\fix f$ is the unique nonempty fixed point class of $f$.

Note that all liftings of $f$ are of the form $(u, v)\tilde f$, where $(u,v)\in \pi_1(\Sigma_2)\times \pi_1(T^2)$ is regarded as the deck transformation on the universal covering space $\widetilde E=\mathbb{H}^2\times \mathbb{R}^2$ of $E$. It obvious that $\fix((u, v)\tilde f) = \emptyset$ if $u\ne 1$. In order to prove $f$ has unique nonempty fixed point class, which is determined by $\tilde f$, it is sufficient to show that for any $v', v''\in \pi_1(T^2)$, there is an element $(u,v)\in \pi_1(\Sigma_2)\times \pi_1(T^2)$ such that
$$(1, v'') = (u, v)^{-1} (1, v') f_\pi(u,v),$$
i.e.
$$(1, v'') =  (u^{-1}, v^{-1}) (1, v') (u, r_\pi(u)\xi(v))\in \pi_1(\Sigma_2)\times \pi_1(T^2).$$
This is equivalent to say
$$
v'' = v^{-1} v' r_\pi(u) \xi(v).
$$
Assume that $v' =a^{s'}b^{t'}$, $v'' =a^{s''}b^{t''}$. Then we can take
$$u=a_1^{s''-s'}\in\pi_1(\Sigma_2), \quad v=a^{t''-t'}b^{t'-t''}\in\pi_1(T^2),$$
and hence
$$
\begin{array}{rcl}
v^{-1} v' r_\pi(u) \xi(v)
& = &a^{t'-t''}b^{t''-t'}\cdot a^{s'}b^{t'} \cdot a^{s''-s'}
   \cdot (a^{m+1}b)^{t''-t'}(a^m b)^{t'-t''}\\
& = & a^{s''}b^{t''}\\
& = & v''.
\end{array}
$$
Thus, we are done.

(4) Note that for a fixed $x\in \Sigma_2$, the restriction $f_x: T^2\to T^2$ of $f$ on the fiber $p^{-1}(x)\cong T^2$ is a homeomorphism which induces the automorphism
$$(f_x)_\pi=\xi: \pi_1(T^2)\to \pi_1(T^2).$$
Hence,
$$L(f_x) = \det
(\left(
\begin{array}{ll}
  1 & 0 \\
  0 & 1  \\
\end{array}
\right)
-
\left(
\begin{array}{ll}
  m+1 & m \\
  1 & 1  \\
\end{array}
\right)
)
= - m,
$$
$$N(f_x)  =|L(f_x) |=m.$$
By Step (3), all of these $m$ fixed point classes lie in the same fixed point class $\F=\fix f$, which is the unique essential fixed point class of $f$, having index
\begin{equation}\label{eq.}
\ind(f,\F)=L(f) =L(\bar f)L(f_x) = (-2)\times (-m) = 2m,
\end{equation}
where the second ``$=$" follows from \cite[IV 3.6 Theorem]{J83}, and $\bar f: \Sigma_2\to \Sigma_2$ is the identity.

Since the integer $m\geq 1$ can be arbitrarily large, Equation (\ref{eq.}) implies that $E$ does not have BIPH.
\end{proof}

\begin{rem}
(1) In the above proof, if $\tilde r(z) ={{\tilde r_x(z)}\choose{\tilde r_y(z)}}\in \mathbb{R}^2$, then
$$\fix(\tilde f)=\left\{(z, ~{{-\tilde r_y(z)}\choose{\tilde r_y(z)-\frac{1}{m}\tilde r_x(z)}})~\bigg|~
z\in \mathbb{H}^2\right\}$$
is a connected set;

(2) if Equation (\ref{fuza chanrao}) is taken as Equation (\ref{eq-lift1}):
$$
r_\pi(a_1) = a,\ r_\pi(b_1) =r_\pi(a_2) =r_\pi(b_2) = 1,
$$ then the computation of $\ind(f, \F)$ will be simpler, but for the sake of diversity, we have retained the above proof.
\end{rem}

\section{Alternative proof of Equation (\ref{eq.})}
In this section, we give an alternative proof of the key Equation (\ref{eq.}).

Let's use the same notations in Section \ref{Proof of Theorem main thm2}. Then by Step (1),
$$
\begin{array}{rcl}
\fix f_\pi
& = & \{ (u,v)\in \pi_1(\Sigma_2)\times \pi_1(T^2) \mid (u, v)=(u, r_\pi(u)\xi(v)) \} \\
& = & \{ (u,v)\in \pi_1(\Sigma_2)\times \pi_1(T^2) \mid  v = r_\pi(u)\xi(v) \}. \\
\end{array}
$$
Take $v=a^sb^t$ and $r_\pi(u)= a^kb^l$. The equality $v = r_\pi(u)\xi(v)$ yields
$$a^sb^t = a^k b^l (a^{m+1}b)^s(a^m b)^t \in \pi_1(T^2).$$
Since $\pi_1(T^2)$ is abelian group with generators $a,b$, we obtain that
$$s = k+ (m+1)s+ mt, \ \  l+s+t =t.$$
It follows that
$$ s= - l, \ \ t= l-\frac{k}{m}.$$
Note that $k$ is actually the sum of power $\nu(u, a_1)$ of $a_1$ in $u$. Thus,
$$
\fix f_\pi
=\left\{ (u,v)\in \pi_1(\Sigma_2)\times \pi_1(T^2) ~\Big|~
   \begin{array}{l}
   v=a^{-\nu(u, b_1)} b^{\nu(u, b_1)-\nu(u, a_1)/m}, \\
   \nu(u, a_1)\equiv 0\mod m
   \end{array}
 \right\}.
$$
Hence, consider homomorphism  $p_\pi: \pi_1(E)\to \pi_1(\Sigma_2)$, we have
$$
p_\pi(\fix f_\pi)
=\{ u \in \pi_1(\Sigma_2) \mid \nu(u, a_1)\equiv 0\mod m  \}.
$$

Moreover, note that $\bar f_\pi: \pi_1(\Sigma_2)\to \pi_1(\Sigma_2)$ is the identity, and $\nu(u, a_1)$ gives an epimorphism
$\pi_1(\Sigma_2)\to \Z$. We have that
$$[\fix \bar f_\pi: p_\pi(\fix f_\pi)]=[\Z :m\Z]=m.$$
Therefore, by Lemma \ref{lem}, the fixed point class $\F=\fix f$ of $f$ with index
$$|\ind(f,\F)|=[\fix\bar f_{\pi} : p_{\pi}(\fix f_{\pi})]\cdot|\ind(\bar f, p(\F)|=m\cdot|\chi(\Sigma_2)|=2m,$$
where $ p(\F)=\Sigma_2$ and $\bar f: \Sigma_2\to \Sigma_2$ is the identity. Thus, we are done.\\

\vspace{6pt}

\noindent\textbf{Acknowledgements.} The first author thanks Shengkui Ye and Bin Yu for helpful communications.


\end{document}